\documentclass[reqno]{amsart}
\usepackage{amsmath}
\usepackage{amscd,amsthm,amsfonts,amsopn,amssymb}
\usepackage{mathrsfs,color}

\newtheorem{theorem}{Theorem}[section]

\newtheorem{proposition}[theorem]{Proposition}

\newtheorem{lemma}[theorem]{Lemma}

\theoremstyle{remark}
\newtheorem{remark}[theorem]{Remark}

\DeclareMathOperator{\mes}{mes\,}
\DeclareMathOperator{\dist}{dist}

\allowdisplaybreaks
\begin{document}

\title[Almost sure global well posedness for 2D radial NLS on the ball]{Almost sure global well posedness for the radial nonlinear Schr\"odinger equation on the unit ball I: the 2D case}
\date{\today}
\author{Jean Bourgain}
\address{(J. Bourgain) School of Mathematics, Institute for Advanced Study, Princeton, NJ 08540}
\email{bourgain@math.ias.edu}
\author{Aynur Bulut}
\address{(A. Bulut) School of Mathematics, Institute for Advanced Study, Princeton, NJ 08540}
\email{abulut@math.ias.edu}
\begin{abstract}
Our first purpose is to extend the results from \cite{T} on the radial defocusing NLS on the disc in $\mathbb{R}^2$ to arbitrary smooth (defocusing) nonlinearities and show the existence of a well-defined flow on the support of the Gibbs measure (which is the natural extension of the classical flow for smooth data).  We follow a similar approach as in \cite{BB-1} exploiting certain additional a priori space-time bounds that are provided by the invariance of the Gibbs measure.

Next, we consider the radial focusing equation with cubic nonlinearity (the mass-subcritical case was studied in \cite{T2}) where the Gibbs measure is subject to an $L^2$-norm restriction.  A phase transition is established, of the same nature as studied in the work of Lebowitz-Rose-Speer \cite{LRS} on the torus.  For sufficiently small $L^2$-norm, the Gibbs measure is absolutely continuous with respect to the free measure, and moreover we have a well-defined dynamics.
\end{abstract}
\maketitle

\section{Introduction}

The purpose of this work is to establish global well-posedness results for the initial value problems associated to the defocusing ($-$) and focusing ($+$) nonlinear Schr\"odinger equation,
\begin{align}
\left\lbrace\begin{array}{rl}iu_t+\Delta u\mp |u|^{\alpha}u&=0\\
u|_{t=0}&=\phi.
\end{array}\right.\label{equation}
\end{align}
with $\alpha\in 2\mathbb{N}$ in the defocusing case, posed on the two-dimensional unit ball $B_2\subset\mathbb{R}^2$, and with $\alpha=\frac{4}{d}$ in the focusing case, posed on the $d$-dimensional unit ball $B_d\subset \mathbb{R}^d$, $d\geq 2$.  In both cases, we prescribe Dirichlet boundary conditions $u(t)=0$ on $\partial B_d$ for all $t\in\mathbb{R}$.  

In order to obtain results globally in time we will appeal to a probabilistic viewpoint, invoking the construction of an invariant Gibbs measure developed in the setting of nonlinear dispersive equations in the works \cite{B1,B2,B3}.  To motivate our discussion below, let us first recall that a Hamiltonian system of the form 
\begin{align}
\frac{d}{dt}\left[\begin{array}{c}p_i\\q_i\end{array}\right]_{i=1,\cdots,n}&=\left[\begin{array}{c}\partial H/\partial q_i\\\partial H/\partial p_i\end{array}\right]_{i=1,\cdots n}\label{fd-hamiltonian}
\end{align}
with $H=H(p_1,p_2,\cdots, p_n,q_1,q_2,\cdots, q_n)$ is subject to the following invariance property: the {\it Gibbs measure} 
\begin{align*}
e^{-H(p, q)}d\mathcal{L}^{2n}(p,q)
\end{align*}
satisfies
\begin{align}
\nonumber &\int_A e^{-H(p,q)}d\mathcal{L}^{2n}(p,q)\\
&\hspace{0.6in}=\int_{\{(p(t),q(t)):(p(0),q(0))\in A\}} e^{-H(p(t),q(t))}d\mathcal{L}^{2n}(p(t),q(t))\label{eq-invariance-1}
\end{align}
for every measurable set $A\subset \mathbb{R}^{2n}$ and $t\in\mathbb{R}$, where $\mathcal{L}^{2n}$ denotes Lebesgue measure and we use the abbreviation $p=(p_1,\cdots, p_n)$, $q=(q_1,\cdots,q_n)$.  

The relevance of this observation to our present study is that the equation in ($\ref{equation}$) is of the form $iu_t=\partial H/\partial \overline{u}$, with conserved Hamiltonian
\begin{align*}
H(\phi)=\frac{1}{2}\int_B |\nabla \phi|^2dx\pm \frac{1}{\alpha+2}\int_B |\phi|^{\alpha+2}dx.
\end{align*}

In order to access the invariance ($\ref{eq-invariance-1}$) of the Gibbs measure in this infinite dimensional setting, we shall consider a sequence of finite-dimensional projections of the problem ($\ref{equation}$), namely
\begin{align}
\left\lbrace\begin{array}{rl}iu_t+\Delta u\mp P_N(|u|^\alpha u)&=0\\
u|_{t=0}&=P_N\phi
\end{array}\right.\label{truncated}
\end{align}
for every integer $N\geq 1$, where $P_N$ denotes the frequency truncation operator defined via the relation
\begin{align*}
P_N\bigg(\sum_{n\in\mathbb{N}} a_ne_n(x)\bigg)=\sum_{\{n\in\mathbb{N}:z_n\leq N\}} a_ne_n(x),
\end{align*}
with $(e_n)$ the sequence of radial eigenfunctions and $(z_n^2)$ the sequence of associated eigenvalues of $-\Delta$ with vanishing Dirichlet boundary conditions.  

Solutions $u_N$ to ($\ref{truncated}$) exist globally in time and can be represented as
\begin{align*}
u_N(t,x)=\sum_{\{n\in\mathbb{N}:z_n\leq N\}} u_n(t)e_n(x),
\end{align*}
and the equation may be written in the form ($\ref{fd-hamiltonian}$) with
\begin{align*}
p_i=\textrm{Re}(u_n(t)),\quad q_i=\textrm{Im}(u_n(t)).
\end{align*}
The Hamiltonian associated to the finite-dimensional projected problem ($\ref{truncated}$) is then
\begin{align*}
H_N(\phi)=\frac{1}{2}\sum_{z_n\leq N} z_n^2|\widehat{\phi}(n)|^2\pm \frac{1}{\alpha+2}\int_B |P_N\phi(x)|^{\alpha+2}dx.
\end{align*}
Furthermore, the flow map
\begin{align*}
\phi_N\mapsto u_N(t)
\end{align*}
leaves invariant the Gibbs measure $\mu_G$ corresponding to ($\ref{truncated}$) defined by
\begin{align}
d\mu_G&=e^{-H_N(\phi))}d\phi=e^{\mp \frac{1}{\alpha+2}\lVert P_N\phi\rVert_{L_x^{\alpha+2}}^{\alpha+2}}d\mu_F^{(N)},\label{gibbs-truncated}
\end{align}
with $\mu_F^{(N)}$ denoting the free probability measure induced by the mapping
\begin{align*}
\omega\mapsto \frac{1}{\pi}\sum_{\{n\in\mathbb{N}:z_n\leq N\}} \frac{g_n(\omega)}{z_n}e_n,\quad \omega\in\Omega
\end{align*}
where %$(z_n^2)$, $(e_n)$ are the sequences of eigenvalues and associated eigenfunctions of $-\Delta$ on the ball $B$, and 
$(g_n)$ is a sequence of normalized independent Gaussian random variables on a probability space $(\Omega, p,\mathcal{M})$.

As noted in \cite{B1,B-PCMI}, when writing (\ref{gibbs-truncated}) one must take care to ensure (i) the $\mu_F$-a.s. existence of the norm $\lVert P_N\phi\rVert_{L_x^{\alpha+2}}$ and (ii) the integrability of the density $e^{\mp\frac{1}{\alpha+2}\lVert P_N\phi\rVert_{L_x^{\alpha+2}}^{\alpha+2}}$ with respect to the measure $\mu_F$.  In both the defocusing and focusing cases, the first condition is satisfied as a consequence of estimates on the eigenfunctions.  On the other hand, while the second condition is trivial in the defocusing setting, it is not satisfied in general when focusing interactions are present.

In the setting of focusing periodic NLS on the one-dimensional torus, the non-integrability of the density was overcome in the work of Lebowitz-Rose-Speer \cite{LRS} by restriction to a ball in the conserved $L_x^2$ norm.  In particular, one fixes $\rho>0$ and considers
\begin{align*}
d\mu_G= e^{\frac{1}{\alpha+2}\lVert P_N\phi\rVert_{L_x^{\alpha+2}}^{\alpha+2}}\chi_{\{\lVert P_N\phi\rVert_{L_x^2}<\rho\}}(\phi)d\mu_F^{(N)}
\end{align*}
which is again invariant under the evolution and can be normalized to a well-defined measure for all $\alpha\leq 4$ (the case $\alpha=4$ requires $\rho$ sufficiently small).

\subsection{Main results of the present work}

In recent works \cite{BB-1,BB-3} (see also \cite{BB}), we addressed the Cauchy problems corresponding to ($\ref{equation}$) for the defocusing nonlinear wave and Schr\"odinger equations on the unit ball of $\mathbb{R}^3$, establishing global well-posedness results for radial solutions with random initial data according to the support of the Gibbs measure (almost surely in the randomization).  

We say that functions $u,u_N:I\times B_d\rightarrow\mathbb{C}$ are solutions of ($\ref{equation}$), ($\ref{truncated}$), respectively, if they belong to the class $C_t(I;H_x^\sigma(B_d))$ for some $\sigma<\frac{1}{2}$ and satisfy the associated integral equations
\begin{align}
u(t)=\phi\pm i\int_0^t e^{i(t-\tau)\Delta}[|u(\tau)|^\alpha u(\tau)]d\tau,\quad t\in I\label{duhamel-equation}
\end{align}
and
\begin{align}
u_N(t)=P_N\phi\pm i\int_0^t e^{i(t-\tau)\Delta}P_N[|u(\tau)|^\alpha u(\tau)]d\tau,\quad t\in I.\label{duhamel-truncated}
\end{align}

To proceed with our discussion, recall that we consider the sequence of finite-dimensional projections ($\ref{truncated}$).  Our estimates will typically be uniform in the truncation parameter $N$.  To accomodate this, we will often make use of the probability measure $\mu_F$ %on $L_x^2(B)$ defined as the measure 
induced by the mapping
\begin{align*}
\omega\mapsto \phi^{(\omega)}:=\frac{1}{\pi}\sum_{n\in \mathbb{N}} \frac{g_n(\omega)}{z_n}e_n.
\end{align*}

Note that with this notation, one has $\mu_F^{(N)}=P_N[\mu_F]$.  Moreover, for each $N\geq 1$, the support of the Gibbs measure $\mu_G$ corresponds to the set
\begin{align*}
\{P_N\phi^{(\omega)}:\omega\in\Omega\}.
\end{align*}

With this probabilistic framework in mind, our first main result, concerning the defocusing problem, takes the following form:
\begin{theorem}
\label{thm-nls-2d}
Fix $\alpha\in 2\mathbb{N}$.  With the above notations, for $N\in \mathbb{N}$, $\omega\in \Omega$, let $u_N$ denote the solution to (\ref{truncated}) in the defocusing case on the two-dimensional unit ball with initial data $P_N\phi=P_N\phi^{(\omega)}$.  Then almost surely in $\Omega$, for every $0<T<\infty$ there exists $u_*\in C_t([0,T);H_x^s(B_2))$, $s<\frac{1}{2}$ such that $u_N$ converges to $u_*$ in $C_t([0,T);H_x^s(B_2))$.
\end{theorem}

We remark that Theorem $\ref{thm-nls-2d}$ was announced in \cite{BB}.  The restriction on the nonlinearity to $\alpha\in 2\mathbb{N}$ is by no means essential, and serves only to simplify the estimates on the nonlinearity, avoiding technicalities due to fractional powers.  The case $\alpha<4$ was treated in \cite{T}.

The proof of Theorem $\ref{thm-nls-2d}$ further develops the method of \cite{BB-1} and consists of an analysis of convergence properties of solutions to the truncated equations ($\ref{truncated}$).  In order to perform this analysis, we will make use of three key ingredients:
\begin{itemize}
\item[(i)] A detailed study of embedding properties associated to the Fourier restriction spaces $X^{s,b}$ (see Lemma $\ref{lem26}$ and Lemma $\ref{lem35}$),
\item[(ii)] A probabilistic estimate demonstrating how the randomization procedure leads to additional $L_x^pL_t^q$  control, almost surely in the probability space (see Proposition $\ref{prop-prob}$), and
\item[(iii)] A bilinear estimate of the nonlinearity enabling one to estimate interactions of high and low frequencies, allowing for a paraproduct-type analysis in the present setting (see Proposition $\ref{propP0}$).
\end{itemize}

The embeddings established in Lemma $\ref{lem26}$ and Lemma $\ref{lem35}$ use frequency decomposition techniques, exploiting the product structure inherent in the $L_t^4$ norm and the Plancherel identity.  A technical tool used to estimate the frequency interactions at this stage are arithmetic estimates for the counting of lattice points on circles (see in particular Lemma $\ref{lem-nt}$).

On the other hand, the probabilistic $L_x^pL_t^q$ bounds of Proposition $\ref{prop-prob}$ make essential use of the fact that $u_N$ is a solution of ($\ref{truncated}$).  More precisely, the improvement in integrability follows from the invariance of the Gibbs measure and bounds for functions belonging to its support.

Turning to the bilinear estimate Proposition $\ref{propP0}$, the $X^{s,b}$ norm of certain products in the Duhamel formula are estimated by $X^{s,b}$ and $L_{t}^2H_x^{\gamma}$ norms  ($\gamma>0$ small) of its factors -- this involves appropriate high and low frequency localizations.  The proof of this proposition is in the flavor of similar estimates in the $\mathbb{R}^d$ setting, with the additional component that the usual convolution identities are replaced with estimates on the correlation of eigenfunctions.

To conclude the proof of Theorem $\ref{thm-nls-2d}$, the ingredients (i), (ii) and (iii) are combined in order to show that the approximate solutions $u_N$ almost surely converge in the space $X^{s,b}$ via a bootstrap-type argument.  We refer the reader to Section $5$ for the full details of the argument.

Our second main result, treating the focusing problem, is as follows. 
\begin{theorem}
\label{thm-focusing}
Set $\alpha=2$.  For each $N\in \mathbb{N}$, $\omega\in \Omega$ let $u_N$ denote the solution to (\ref{truncated}) in the focusing case on the two dimensional unit ball with initial data $P_N\phi=P_N\phi^{(\omega)}$ and subject to an appropriate $L^2$-norm restriction.  Then for every $0<T<\infty$, there exists almost surely $u_*\in C_t([0,T);H_x^s(B_2))$, $s<\frac{1}{2}$, such that $u_N$ converges to $u_*$ in $C_t([0,T);H_x^s(B_2))$.
\end{theorem}

The main additional issue in the proof of Theorem $\ref{thm-focusing}$ is to show the $\mu_F$-integrability of the map
\begin{align*}
\phi\mapsto e^{\frac{1}{\alpha+2}\lVert \phi\rVert_{L_x^{\alpha+2}}^{\alpha+2}}\chi_{\{\lVert P_N\phi\rVert_{L_x^2}<\rho\}}(\phi)
\end{align*}
provided that $\rho>0$ is chosen sufficiently small.  This result is stated in Proposition $\ref{prop-focusing}$, and ensures a bound on the $L_x^2$-truncated Gibbs measures 
\begin{align*}
d\mu_G^{(N)}=e^{\frac{1}{\alpha+2}\lVert \phi\rVert_{L_x^{\alpha+2}}^{\alpha+2}}\chi_{\{\lVert P_N\phi\rVert_{L_x^2}<\rho\}}(\phi)d\mu_F^{(N)}
\end{align*}
Once we have the invariant measure at our disposal, the convergence of the solutions of the truncated equations follows from the same argument as in the defocusing case for $\alpha=2$, leading to a well-defined dynamics on the support of the modified Gibbs measure.

For subscritical nonlinearity $\alpha<2$, the corresponding result was established in \cite{T2} (with arbitrary $L^2$-truncation).

In Remark \ref{remark-s5} in $\S5$, we will also comment on what happens for larger $L^2$-norm restriction $\rho$.

\subsection*{Outline of the paper}

The remainder of this paper is structured as follows: in Section $2$ we establish our notation and recall the definitions of the function spaces which will be used in the remainder of the paper.  Section $3$ is then devoted to the proof of a probabilistic estimate for solutions corresponding to initial data in the support of the Gibbs measure.  In Section $4$, we establish a key bilinear estimate on the nonlinearity, while the proof of Theorem $\ref{thm-nls-2d}$ is contained in Section $5$.  We conclude by establishing Theorem $\ref{thm-focusing}$ in Section $6$.

\section{Preliminaries}

\subsection{Notation}
Let us now establish some brief notational conventions.  Unless otherwise indicated, we will use the conventions $n\in\mathbb{N}$, $m\in\mathbb{Z}$, while capital letters $K$, $N$ and $M$ shall denote dyadic integers of the form $2^k$, $k\geq 0$.  Throughout our arguments we will frequently make use of a dyadic decomposition in frequency, writing
\begin{align*}
f(x)=\sum_n \hat{f}(n)e_n(x)=\sum_{N\geq 1}\sum_{n\sim N} \hat{f}(n)e_n(x),
\end{align*}
where for each dyadic integer $N$, the condition $n\sim N$ is characterized by $N\leq n<2N$ (likewise, we say $m\sim M$ if $M\leq |m|\leq 2M$).  We shall also use the notation $\langle x\rangle=(1+|x|^2)^{1/2}$.

For each $n\in\mathbb{N}$, let $z_n\in\mathbb{R}\setminus \{0\}$ be such that $z_n^2$ is the $n$th eigenvalue of the Dirichlet Laplacian on $B_2$, and recall that $z_n$ satisfies
\begin{align}
z_n=\pi\left(n-\tfrac{1}{4}\right)+O\left(\tfrac{1}{n}\right).\label{eq_asymp}
\end{align}
Following the usual convention, we will often refer to $(z_n)$ as the sequence of frequencies for functions defined on the ball $B_2$.  Moreover, let $e_n$ denote the $n$th radial eigenfunction, corresponding to the eigenvalue $z_n^2$.  One then has
\begin{align}
\begin{array}{ll}\lVert e_n\rVert_{L_x^p}\lesssim 1,&p\in [2,4),\\
\lVert e_n\rVert_{L_x^p}\lesssim \log(2+n)^{1/4},&p=4,\\
\lVert e_n\rVert_{L_x^p}\lesssim n^{\frac{1}{2}-\frac{2}{p}},&p\in (4,\infty).\end{array}\label{efn-bds}
\end{align}

We now state a basic probabilistic estimate for Gaussian random variables.  In particular, if $(g_n)$ is a sequence of independent (normalized) complex Gaussians, then we have
\begin{align}
\bigg\lVert \sum_{n} \alpha_ng_n(\omega)\bigg\rVert_{L^q(d\omega)}&\lesssim \sqrt{q}\big(\sum_{n} |\alpha_n|^2\big)^{1/2}.\label{prob-est}
\end{align}

Moreover, if $X(\omega)$ is a Gaussian process with values in some normed space $(E,\lVert\cdot\rVert)$, of finite expectation $\mathbb{E}_\omega\left[\,\lVert X\rVert\,\right]$, it follows that
\begin{align*}
\int e^{c\left(\frac{\lVert X\rVert}{\mathbb{E}[\lVert X\rVert]}\right)^2}<C
\end{align*}
and hence
\begin{align}
\mathbb{P}_\omega\Big[\,\lVert X\rVert>t\,\mathbb{E}_\omega\left[\,\lVert X\rVert\,\right]\,\Big]\lesssim e^{-ct^2},\quad t>1.\label{f-8}
\end{align}

The results and analysis in this section appear basically in \cite{T} and are repeated here in a form suitable for our presentation and in the interest of being self-contained.

\subsection{Arithmetic estimates}

As usual in the study of nonlinear Schr\"odinger equations on bounded domains (e.g. the case of tori treated in \cite{B1}, \cite{B2}) an essential component of our analysis will rely upon arithmetical bounds for the sequence of frequencies.  In particular, we shall use the following:
\begin{lemma}
\label{lem-nt}
There exists $c>0$ such that for every $R>R_1\gg 1$ and all boxes $Q\subset\mathbb{R}^2$ of size $R_1$, we have
\begin{align}
\Big|\Big\{(n_1,n_2)\in\mathbb{Z}^2:n_1^2+n_2^2=R^2, (n_1,n_2)\in Q\Big\}\Big|\lesssim \exp\bigg(c\frac{\log R_1}{\log\log R_1}\bigg)\label{jarnick-1}
\end{align}
\end{lemma}
\begin{proof}
Let $R_1<R$ be given.  Suppose first that $R_1$ and $R$ satisfy $R^{1/3}\lesssim R_1$.  Note that factorization in the Gaussian integers $\mathbb{Z}+i\mathbb{Z}$ implies the bound
\begin{align}
\nonumber \Big|\Big\{ (n_1,n_2)\in\mathbb{Z}^2:n_1^2+n_2^2=R^2\Big\}\Big|&\leq \exp\Big(\Big|\Big\{\textrm{Gaussian prime factors of}\, R^2\Big\}\Big|\Big)\\
&<\exp\,\frac{\log R}{\log\log R}.\label{eq-AABA1}
\end{align}
Since $\log R\sim\log R_1$, the inequality ($\ref{jarnick-1}$) now follows from ($\ref{eq-AABA1}$).

On the other hand, suppose that $R_1\lesssim R^{1/3}$.  It then follows by Jarnick's theorem for lattice points on circles that the left-hand side of ($\ref{jarnick-1}$) is at most $2$, which gives the claim in this case.
\end{proof}

\begin{lemma}
\label{lem-nt-2}
Let $z_n^2$ be the $n$th eigenvalue of the Dirichlet Laplacian on $B_2$ and let $Q\subset\mathbb{R}^2$ be a box of size $R_1$.  Then for any $\ell\in \mathbb{R}_+$,
\begin{align*}
\left|\left\{(n_1,n_2)\in \mathbb{Z}^2:|z_{n_1}^2+z_{n_2}^2-\ell|<1,\,(n_1,n_2)\in Q\right\}\right|\lesssim \exp\left(c\frac{\log R_1}{\log\log R_1}\right).
\end{align*}
\end{lemma}

\begin{proof}
According to (\ref{eq_asymp}), $z_n^2=\pi^2\left(n-\frac{1}{4}\right)^2+O(1)$.  Therefore the eqaution $|z_{n_1}^2+z_{n_2}^2-\ell|<1$ implies
\begin{align*}
\left|\pi^2\left(n_1-\frac{1}{4}\right)^2+\pi^2\left(n_2-\frac{1}{4}\right)^2-\ell\right|&<O(1)\\
\left|(4n_1-1)^2+(4n_2-1)^2-\frac{16\ell}{\pi^2}\right|&<O(1)
\end{align*}
and we can apply Lemma \ref{lem-nt} setting $n'_1=4n_1-1$, $n'_2=4n_2-1$.
\end{proof}

\subsection{Description of the $X^{s,b}$ spaces}

Fix $I=[0,T)$ with $0<T<\frac{1}{2}$, and let $X^{s,b}(I)$ denote the class of functions $f:I\times B\rightarrow\mathbb{C}$ representable as
\begin{align}
f(t,x)=\sum_{n,m} f_{n,m}e_n(x)e(mt),\quad (t,x)\in I\times B\label{s2e1}
\end{align}
for which the norm
\begin{align*}
\lVert f\rVert_{{s,b}}:=\inf \,\bigg(\sum_{n,m} \langle z_n\rangle^{2s}\langle z_n^2-m\rangle^{2b}|f_{n,m}|^2\bigg)^{1/2}
\end{align*}
is finite, with the infimum taken over all representations ($\ref{s2e1}$).  Throughout the remainder of the paper, we will assume $0<T<\frac{1}{2}$, unless otherwise indicated.

We now give two lemmas expressing some embeddings of the space $X^{s,b}$ which will be essential components of our analysis below.  Similar estimates appear already in \cite{T} (see in particular \cite[Proposition $4.1$]{T}).
\begin{lemma}
\label{lem26}
Let $\frac{1}{4}<b<1$ and $2\leq p<4$ be given.  Then, letting $P_If=\sum_{z_n\in I} \widehat{f}(n)e_n$, we have for $\epsilon>0$, $f\in\mathcal{S}$ and intervals $I\subset \mathbb{R}$,
\begin{align}
\lVert P_If\rVert_{L_x^pL_t^4}\lesssim \left\lbrace\begin{array}{ll}|I|^{\epsilon}\lVert P_If\rVert_{{0,b}}&\quad\textrm{for}\quad b>\frac{1}{2},\\
|I|^{1-2b+\epsilon}\lVert P_If\rVert_{{0,b}}&\quad\textrm{for}\quad b<\frac{1}{2}.{\vbox to 15pt {\vfil \hbox to 1cm{\  }\vfil}}\end{array}\right.
\label{eq-fn-bd-1}
\end{align}
\end{lemma}

\begin{proof}
We begin by establishing the first inequality in ($\ref{eq-fn-bd-1}$), for which we shall compute the norm directly.  

Fix $\epsilon>0$ and write
\begin{align}
\nonumber P_If(t,x)=\sum_{\substack{m\in\mathbb{Z}\\z_n\in I}} \widehat{f}(m,n)e_n(x)e(mt)=\sum_{m}\bigg(\sum_{z_n\in I} \widehat{f}(m+[z_n^2],n)e_n(x)e(z_n^2t)\bigg)e(mt).
\end{align}
Performing a dyadic decomposition into intervals $m\sim M$, we obtain
\begin{align}
\nonumber \lVert P_If\rVert_{L_x^pL_t^4}&\lesssim \sum_M \lVert f_M\rVert_{L_x^pL_t^4}
\end{align}
with
\begin{align}
\nonumber f_M=\sum_{m\sim M} \sum_{z_n\in I} \widehat{f}(m+[z_n^2],n)e_n(x)e(z_n^2t)e(mt).
\end{align}

We then have
\begin{align}
\nonumber &\lVert f_M\rVert_{L_x^pL_t^4}\\
\nonumber &\hspace{0.2in}\lesssim \sum_{m\sim M} \bigg\lVert \sum_{\ell}\sum_{\substack{z_n,z_{n'}\in I\\|z_n^2+(z_{n'})^2-\ell|<1}} \widehat{f}(m+[z_n^2],n)\,\widehat{f}(m+[(z_{n'})^2],n')\, e_n(x)e_{n'}(x)e^{i\ell t}\bigg\rVert_{L_x^{p/2}L_t^2}^{1/2}\\
\nonumber &\hspace{0.2in}\lesssim \sum_{m\sim M} \bigg\lVert \bigg(\sum_{\ell}\bigg|\sum_{\substack{z_n,z_{n'}\in I\\|z_n^2+(z_{n'})^2-\ell|<1}} \widehat{f}(m+[z_n^2],n)\\
&\hspace{1.8in}\cdot\widehat{f}(m+[(z_{n'})^2],n')e_n(x)e_{n'}(x)\bigg|^2\bigg)^{1/2}\bigg\rVert_{L_x^{p/2}}^{1/2},\label{eq-cc-2}
\end{align}
where in obtaining the last inequality we have used the Plancherel identity in the $t$ variable.

Using the Cauchy-Schwarz inequality and Lemma $\ref{lem-nt-2}$,
\begin{align*}
(\ref{eq-cc-2})&\lesssim  \sum_{m\sim M} \bigg(\sup_{\ell}\sum_{\substack{z_n,z_{n'}\in I\\|z_n^2+(z_{n'})^2-\ell|<1}}1\bigg)^{1/4}\bigg\lVert \bigg(\sum_{z_n\in I} |\widehat{f}(m+[z_n^2],n)|^2e_n(x)^2\bigg)\bigg\rVert_{L_x^{p/2}}^{1/2}\\
&\lesssim \sum_{m\sim M} |I|^{\epsilon} \bigg(\sum_{z_n\in I} |\widehat{f}(m+[z_n^2],n)|^2\lVert e_n(x)\rVert_{L_x^{p}}^2\bigg)^{1/2}\\
&\lesssim \sum_{m\sim M} |I|^{\epsilon} \bigg(\sum_{z_n\in I} |\widehat{f}(m+[z_n^2],n)|^2\bigg)^{1/2}
\end{align*}
where to obtain the last inequality we used the eigenfunction estimate ($\ref{efn-bds}$).

Invoking the Cauchy-Schwarz inequality once more, 
\begin{align}
\nonumber (\ref{eq-cc-2})&\lesssim |I|^{\epsilon}M^{\frac{1}{2}}\bigg(\sum_{m\sim M} \sum_{z_n\sim I} |\widehat{f}(m+[z_n^2],n)|^2\bigg)^{1/2}\\
\nonumber &= |I|^{\epsilon}M^{\frac{1}{2}}\bigg(\sum_{\substack{z_n\sim I\\m-z_n^2\sim M}}  |\widehat{f}(m,n)|^2\bigg)^{1/2}\\
&\lesssim |I|^{\epsilon}M^{\frac{1}{2}-b}\lVert P_If\rVert_{{0,b}}\label{eq-cc-1}
\end{align}

On the other hand to obtain the second inequality in (\ref{eq-fn-bd-1}), a similar calculation yields
\begin{align}
\nonumber \lVert f_M\rVert_{L_x^pL_t^4}&\leq \sum_{z_n\in I} \bigg\lVert \sum_{m\sim M} \widehat{f}(m+[z_n^2],n)e(mt)\bigg\rVert_{L_t^4}\lVert e_n(x)\rVert_{L_x^p}\\
\nonumber &\leq M^{\frac{1}{4}}\sum_{z_n\in I} \bigg(\sum_{m-z_n^2\sim M} |\widehat{f}(m,n)|^2\bigg)^{1/2}\\
&\leq |I|^{1/2}M^{\frac{1}{4}-b}\lVert P_If\rVert_{{0,b}}\label{eq-cc-1b}
\end{align}

Thus from ($\ref{eq-cc-1}$), ($\ref{eq-cc-1b}$) 
\begin{align*}
\lVert f_M\rVert_{L_x^pL_t^4}\lesssim \min\{|I|^\epsilon M^{\frac{1}{2}-b},|I|^{\frac{1}{2}}M^{\frac{1}{4}-b}\}\lVert P_If\rVert_{0,b}
\end{align*}
and summation in $M$ gives the desired estimates.
\end{proof}

\begin{remark}
\label{rem1a}
The estimates obtained in Lemma $\ref{lem26}$ also allow us to conclude
\begin{align}
\nonumber \lVert f\rVert_{L_x^pL_t^4}&\leq C\lVert f\rVert_{{\epsilon,b}}\quad\textrm{for}\quad b\geq 1/2\,\,\textrm{and}\,\, p\leq 4 %\label{eq-fn-bd-2}
\end{align}
and
\begin{align}
\nonumber \lVert f\rVert_{L_x^pL_t^4}&\leq C\lVert f\rVert_{{1-2b+\epsilon,b}}\quad\textrm{for}\quad 1/4<b<1/2. %\label{eq-fn-bd-3} 
\end{align}

Indeed, appealing to the decomposition $f=\sum_N P_{[N,2N)}f$ with $N\geq 1$ dyadic and applying ($\ref{eq-fn-bd-1}$) on each interval $I=[N,2N)$ (with $\tilde{\epsilon}=\epsilon/2$), we obtain 
\begin{align*}
\lVert f\rVert_{L_x^pL_t^4}&\lesssim \sum_N N^{\epsilon/2}\lVert P_{[N,2N)}f\rVert_{{0,b}}\lesssim \sum_N N^{-\epsilon/2}\lVert f\rVert_{{\epsilon,b}}\lesssim \lVert f\rVert_{{\epsilon,b}}
\end{align*}
for $b\geq 1/2$.  An identical calculation with an application of the second inequality in ($\ref{eq-fn-bd-1}$) in place of the first inequality of ($\ref{eq-fn-bd-1}$) gives the claim for $1/4<b<1/2$.
\end{remark}

As a consequence of Lemma $\ref{lem26}$, we have the following estimates on the nonlinear term of the Duhamel formula ($\ref{duhamel-truncated}$).

\begin{lemma}
\label{lem35}
For each interval $I\subset\mathbb{R}$ and every $b>\frac{1}{2}$, $\epsilon>0$, there exists $C=C(b,\epsilon)>0$ such that
\begin{align}
\bigg\lVert P_I\bigg(\int_0^t e^{i(t-\tau)\Delta}f(\tau)d\tau\bigg)\bigg\rVert_{{0,b}}\leq C|I|^{2b-1+\epsilon}\lVert f\rVert_{L_x^{\frac{4}{3}+\epsilon}L_t^{\frac{4}{3}}}.\label{eq_aa1}
\end{align}

Moreover, the inequality
\begin{align}
\bigg\lVert \int_0^t e^{i(t-\tau)\Delta}f(\tau)d\tau\bigg\rVert_{{0,b}}\leq C\lVert (\sqrt{-\Delta})^{2b-1+\epsilon}f\rVert_{L_x^{\frac{4}{3}+\epsilon}L_t^{\frac{4}{3}}}\label{eq_aa2}
\end{align}
also holds for all $f\in \mathcal{S}$.
\end{lemma}

\begin{proof}
We begin by showing ($\ref{eq_aa1}$), invoking the representation 
\begin{align*}
f(t,x)=\sum_{m,n} \widehat{f}(m,n)e_n(x)e(mt)
\end{align*}
and observing that the left-hand side of ($\ref{eq_aa1}$) is 
\begin{align}
\nonumber &\bigg\lVert \sum_{\substack{m\in \mathbb{Z}\\z_n\in I}}\int_0^t e^{i(t-\tau)z_n^2}\widehat{f}(m,n)e_n(x)e^{im\tau}d\tau\bigg\rVert_{{0,b}}\\
\nonumber &\hspace{1.2in}=\bigg\lVert \sum_{\substack{m\in \mathbb{Z}\\z_n\in I}}\widehat{f}(m,n)e_n(x) \cdot \frac{e(mt)-e(z_n^2t)}{i(m-z_n^2)}\bigg\rVert_{{0,b}}\\
\nonumber &\hspace{1.2in}\lesssim \bigg(\sum_{\substack{m\in\mathbb{Z}\\z_n\in I}} \frac{|\widehat{f}(m,n)|^2}{\langle m-z_n^2\rangle^{2(1-b)}}\bigg)^{1/2}+\bigg(\sum_{z_n\in I} \bigg|\sum_m \frac{\widehat{f}(m,n)}{m-z_n^2}\bigg|^2\bigg)^{1/2}\\
\nonumber &\hspace{1.2in}\lesssim \bigg(\sum_{\substack{m\in\mathbb{Z}\\z_n\in I}} \frac{|\widehat{f}(m,n)|^2}{\langle m-z_n^2\rangle^{2(1-b)}}\bigg)^{1/2}\\
&\hspace{1.2in}=\lVert P_If\rVert_{{0,-(1-b)}}\label{eq_aa2a}
\end{align}
where we have used Cauchy-Schwarz to obtain the second inequality.  

By duality arguments followed by H\"older's inequality combined with Lemma $\ref{lem26}$, we then have 
\begin{align}
\nonumber &\lVert P_If\rVert_{{0,-(1-b)}}\\
\nonumber &\hspace{0.2in}=\sup \bigg\{\bigg|\int P_If(t,x)P_Ig(t,x)dtdx\bigg|:g\in L^2, \lVert P_Ig\rVert_{{0,1-b}}\leq 1\bigg\}\\
\nonumber &\hspace{0.2in}\leq \lVert f\rVert_{L_x^{\frac{4}{3}+\epsilon}L_t^\frac{4}{3}}\lVert P_Ig\rVert_{L_x^{\frac{4+3\epsilon}{1+3\epsilon}}L_t^4}\\
\nonumber &\hspace{0.2in}\lesssim |I|^{1-2(1-b)+\epsilon}\lVert f\rVert_{L_x^{\frac{4}{3}+\epsilon}L_t^\frac{4}{3}}\lVert P_Ig\rVert_{{0,1-b}}\\
&\hspace{0.2in}\lesssim |I|^{2b-1+\epsilon}\lVert f\rVert_{L_x^{\frac{4}{3}+\epsilon}L_t^\frac{4}{3}}\label{eq_aa2b}
\end{align}
The inequality ($\ref{eq_aa1}$) now follows by combining ($\ref{eq_aa2a}$) and ($\ref{eq_aa2b}$).

The proof of ($\ref{eq_aa2}$) proceeds in a similar manner.  Arguing as above and using Remark $\ref{rem1a}$, we obtain
\begin{align}
\nonumber &\bigg\lVert \sum_{\substack{m\in \mathbb{Z}\\n\in \mathbb{N}}}\int_0^t e^{i(t-\tau)z_n^2}\widehat{f}(m,n)e_n(x)e^{im\tau}d\tau\bigg\rVert_{{-(2b-1+\epsilon),b}}\\
\nonumber &\hspace{1.5in}\lesssim\lVert f\rVert_{{-(2b-1+\epsilon),-(1-b)}}\\
\nonumber &\hspace{1.5in}\leq \sup \bigg\{\lVert f\rVert_{L_x^{\frac{4}{3}+\epsilon}L_t^\frac{4}{3}}\lVert g\rVert_{L_x^{\frac{4+3\epsilon}{1+3\epsilon}}L_t^4}:\lVert g\rVert_{{2b-1+\epsilon,1-b}}\leq 1\bigg\}\\
\nonumber &\hspace{1.5in}\lesssim \lVert f\rVert_{L_x^{\frac{4}{3}+\epsilon}L_t^\frac{4}{3}}\lVert g\rVert_{{2b-1+\epsilon,1-b}}\\
\nonumber &\hspace{1.5in}\lesssim \lVert f\rVert_{L_x^{\frac{4}{3}+\epsilon}L_t^\frac{4}{3}}
\end{align}
from which the desired inequality ($\ref{eq_aa2}$) follows immediately.
\end{proof}

\section{Probabilistic estimates}

In this section, we establish a collection of essential probabilistic estimates which will enable us to obtain long-time control over solutions to ($\ref{truncated}$).  We remark that these estimates are uniform in the truncation parameter $N$; this uniformity is important in the convergence proof of the next section.

We begin by establishing an almost sure bound on initial data belonging to the support of the Gibbs measure $\mu_G$.  
\begin{lemma}
\label{lem-prob-init-data}
Fix $s<\frac{1}{2}$.  Then we have the bound
\begin{align}
\mu_F(\{\phi:N_0^{\frac{1}{2}-s}\lVert P_{\geq N_0}\phi\rVert_{H_x^s}>\lambda\})&\lesssim \exp(-\lambda^c)\label{eq-aaaq-1}
\end{align}
for all $N_0\geq 1$ sufficiently large, where $\phi=\phi^{(\omega)}=\sum_{n\in\mathbb{N}} \frac{g_n(\omega)}{z_n}e_n$.
\end{lemma}

\begin{proof}
Fix $q_1\geq 2$ to be determined.  The Tchebyshev and Minkowski inequalities together with the estimate ($\ref{prob-est}$) then imply that the left-hand side of $(\ref{eq-aaaq-1})$ is bounded by
\begin{align}
\nonumber \frac{N_0^{(\frac{1}{2}-s)q_1}}{\lambda^{q_1}}\lVert P_{\geq N_0}\phi\rVert_{L_\omega^{q_1}(d\mu_F;H_x^s)}^{q_1}&\lesssim \frac{N_0^{(\frac{1}{2}-s)q_1}}{\lambda^{q_1}}\bigg\lVert \sum_{n=N_0}^\infty \frac{g_n(\omega)}{z_n^{(1-s)}}e_n\bigg\rVert_{L_x^2(L_\omega^{q_1}(d\mu_F))}^{q_1}\\
\nonumber &\lesssim \bigg(\frac{N_0^{\frac{1}{2}-s}\sqrt{q_1}}{\lambda}\bigg)^{q_1}\bigg(\sum_{n=N_0}^\infty \frac{\lVert e_n(x)\rVert_{L_x^2}^2}{z_n^{2(1-s)}}\bigg)^{q_1/2}\\
\nonumber &\lesssim \bigg(\frac{N_0^{\frac{1}{2}-s}\sqrt{q_1}}{\lambda}\bigg)^{q_1}\bigg( \sum_{n=N_0}^\infty z_n^{-2(1-s)} \bigg)^{q_1/2}\\
&\lesssim \bigg(\frac{\sqrt{q_1}}{\lambda}\bigg)^{q_1},\label{aaaa1}
\end{align}
where the summation in the third line is bounded by $N_0^{-1+2s}$ for $N_0$ sufficiently large, as a consequence of the asymptotic representation ($\ref{eq_asymp}$) for the sequence of eigenvalues $(z_n)$.
Optimizing ($\ref{aaaa1}$) in $q_1$, we obtain the desired estimate ($\ref{eq-aaaq-1}$).
\end{proof}

In particular, we note that Lemma $\ref{lem-prob-init-data}$ includes a description of the decay present when considering the restriction of $\phi$ to high frequencies.  The next proposition is an essential ingredient and combines this type of probabilistic estimate with the invariance of the Gibbs measure $\mu_G$ under the finite-dimensional evolution to obtain certain spacetime bounds of large deviation type.

\begin{proposition}
\label{prop-prob}
Let $T>0$ be given.  Then for every $0\leq \sigma<\frac{1}{2}$, $2\leq p<\frac{2}{\sigma}$ and $q<\infty$, 
\begin{align*}
\mathbb{E}_{\mu_F}\bigg[\lVert (\sqrt{-\Delta})^{\sigma}u_\phi\rVert_{L_x^pL_t^q}\bigg]<C(\sigma,p,q,T)
\end{align*}
where $u_N=u_N^{(\phi)}$ is the solution of ($\ref{truncated}$) corresponding to initial data $P_N\phi$ and the $L_t^q$ norm is taken on the interval $[0,T)$.
In fact, there is the stronger distributional inequality 
\begin{align*}
\mu_F(\{\phi:\lVert (\sqrt{-\Delta})^{\sigma}u_{N}\rVert_{L_x^pL_t^q}>\lambda\})\lesssim \exp(-\lambda^c),\quad \lambda>0,\,N\geq 1,
\end{align*}
for some $c>0$. 
\end{proposition}
\begin{proof}
Fix $\lambda_1>0$ to be determined later in the argument.  Then, denoting $u=u_N^{(\phi)}$, we have
\begin{align}
\nonumber &\mu_F(\{\phi:\lVert (\sqrt{-\Delta})^{\sigma}u\rVert_{L_x^pL_t^q}>\lambda\})\\
\nonumber &\hspace{0.4in}\lesssim e^{\frac{1}{2+\alpha}\lambda_1^{2+\alpha}}\mu_G(\{\phi:\lVert (\sqrt{-\Delta})^{\sigma}u\rVert_{L_x^pL_t^q}>\lambda\})\\
\nonumber &\hspace{0.7in}+\mu_F(\{\phi:\lVert (\sqrt{-\Delta})^{\sigma}u\rVert_{L_x^pL_t^q}>\lambda\}\cap \{\phi:\lVert \phi\rVert_{L_{x}^{2+\alpha}}>\lambda_1\})\\
&\hspace{0.4in}=:\textrm{(I)}+\textrm{(II)}\label{comb--1}
\end{align}

To estimate term (I), we observe that for every $q_1\geq\max\{p,q\}$ the Tchebyshev inequality and Minkowski inequality for integrals allow us to bound
\begin{align}
\nonumber \mu_G(\{\phi:\lVert (\sqrt{-\Delta})^{\sigma}u\rVert_{L_x^pL_t^q}>\lambda\})
\end{align}
by
\begin{align}
\nonumber \frac{1}{\lambda^{q_1}}\int \lVert (\sqrt{-\Delta})^{\sigma}u\rVert_{L_x^pL_t^q}^{q_1}d\mu_G(\phi)
&\lesssim \big(\frac{T^{1/q}}{\lambda}\big)^{q_1}\big\lVert \,\big\lVert (\sqrt{-\Delta})^{\sigma}\phi\big\rVert_{L_x^p}\big\rVert_{L^{q_1}(d\mu_G)}^{q_1}\\
&\lesssim_T \big(\frac{\sqrt{q_1}}{\lambda}\big)^{q_1}\bigg\lVert \bigg(\sum_{z_n\leq N}\frac{|e_n(x)|^2}{z_n^{2(1-\sigma)}}\bigg)^{1/2}\bigg\rVert_{L_x^p}^{q_1}\label{eq-aaa-bbb-1}
\end{align}
where we use the invariance of the Gibbs measure $\mu_G$ under the truncated evolution ($\ref{truncated}$).  Using Minkowski's inequality (since $p\geq 2$), we then have
\begin{align}
(\ref{eq-aaa-bbb-1})&\lesssim \big(\frac{\sqrt{q_1}}{\lambda}\big)^{q_1}\bigg(\sum_{z_n\leq N}\frac{\lVert e_n(x)\rVert_{L_x^p}^2}{z_n^{2(1-\sigma)}}\bigg)^{q_1/2}\label{eq-aaa-bbb-2}
\end{align}
Now, recalling the eigenfunction bounds ($\ref{efn-bds}$), we have the bound
\begin{align}
\sum_{n\in \mathbb{N}} \frac{n^{1-\frac{4}{p}}}{z_n^{2(1-\sigma)}}&\leq C_0+C_1\sum_{n\geq N_0} n^{-(1-2\sigma)-\frac{4}{p}}<\infty,\label{prior-n0}
\end{align}
where we have used the hypothesis $p<\frac{2}{\sigma}$.  

Combining ($\ref{prior-n0}$) with ($\ref{eq-aaa-bbb-2}$) gives
\begin{align}
\mu_G(\{\phi:\lVert (\sqrt{-\Delta})^{\sigma}u\rVert_{L_x^pL_t^q}>\lambda\})&\lesssim \big(\frac{\sqrt{q_1}}{\lambda}\big)^{q_1}\label{bound123123}
\end{align}
where the implicit constant is independent of $N$.  
Optimizing ($\ref{bound123123}$) in $q_1$ then implies the bound
\begin{align}
\textrm{(I)}\lesssim \exp(\frac{1}{2+\alpha}\lambda_1^{2+\alpha})\exp(-\lambda^{2}/(2e)).\label{comb-0}
\end{align}

To estimate term (II), we fix $q_2>2+\alpha$ and again apply the Tchebyshev and Minkowski inequalities to bound the term
\begin{align*}
\mu_F(\{\phi:\lVert \phi\rVert_{L_x^{2+\alpha}}>\lambda_1\})
\end{align*}
by
\begin{align}
\nonumber \frac{1}{\lambda_1^{q_2}}\int \lVert \phi\rVert_{L_{x}^{2+\alpha}}^{q_2}d\mu_F(\phi)
\nonumber &\lesssim \frac{1}{\lambda_1^{q_2}}\big\lVert\,\lVert \phi\rVert_{L^{q_2}(d\mu_F)}\,\big\rVert_{L_{x}^{2+\alpha}}^{q_2}\\
\nonumber &\lesssim \big(\frac{\sqrt{q_2}}{\lambda_1}\big)^{q_2}\bigg\lVert \bigg(\sum_{n} \frac{|e_n(x)|^2}{z_n^2}\bigg)^{1/2}\bigg\rVert_{L_x^{2+\alpha}}^{q_2}\\
\nonumber &\lesssim \big(\frac{\sqrt{q_2}}{\lambda_1}\big)^{q_2}\bigg(\sum_{n} \frac{\lVert e_n(x)\rVert_{L_x^{2+\alpha}}^2}{z_n^2}\bigg)^{q_2/2}.
\end{align}

Appealing to the eigenfunction bounds ($\ref{efn-bds}$) and the asymptotic representation ($\ref{eq_asymp}$) for $z_n$, we estimate
\begin{align*}
\sum_{n\in \mathbb{N}} \frac{n^{1-\frac{4}{2+\alpha}}}{z_n^2}&\leq C_0+\sum_{n\geq N_0} n^{-1-\frac{4}{2+\alpha}}<\infty
\end{align*}
for $N_0$ sufficiently large.  As a consequence we obtain
\begin{align*}
\mu_F(\{\phi:\lVert \phi\rVert_{L_x^{2+\alpha}}>\lambda_1\})&\lesssim \big(\frac{\sqrt{q_2}}{\lambda_1}\big)^{q_2}.
\end{align*}

Minimizing the right hand side over admissible values of $q_2$ we get the bound
\begin{align}
\textrm{(II)}&\lesssim \exp(-\lambda_1^2/(2e)).\label{comb-1}
\end{align}

Combining ($\ref{comb--1}$) with ($\ref{comb-0}$) and ($\ref{comb-1}$) and optimizing in $\lambda_1$ then gives
\begin{align*}
\mu_F(\{\phi:\lVert (\sqrt{-\Delta})^{\sigma}u\rVert_{L_x^pL_t^q}>\lambda\})&\lesssim \exp(-c\lambda^{\frac{4}{2+\alpha}})
\end{align*}
as desired.  This completes the proof of Proposition $\ref{prop-prob}$.
\end{proof}

We will also require a slight refinement of Proposition $\ref{prop-prob}$ which is a consequence of similar arguments, and allows for more precise estimates.
\begin{proposition}
\label{prop-prob-2}
Let $T$, $\sigma$, $p$, $q$ and $(u_N)$  be as in Proposition $\ref{prop-prob}$.  Then for all $M,N\geq 1$ with $M<N$, one has the distributional inequality 
\begin{align*}
\mu_F^{(N)}(\{\phi_N:\lVert (\sqrt{-\Delta})^\sigma (u_N-P_Mu_N)\rVert_{L_x^pL_t^q}>\lambda\})<e^{-(\theta\lambda)^c}
\end{align*}
where we have set $\theta=T^{-\frac{1}{q}}M^{\frac{2}{p}-\sigma}$.
\end{proposition}

\section{Bilinear estimate on the nonlinearity}

We now establish a bilinear estimate on the nonlinear term in the Duhamel formula ($\ref{duhamel-truncated}$) which controls interactions between high and low frequency components of the nonlinearity.

\begin{proposition}
\label{propP0}
Fix $0\leq s\leq 1$, $b>\frac{1}{2}$ and $N\in \mathbb{Z}_+$.  Then for every $\mu>0$ we have the inequality
\begin{align}
\bigg\lVert \int_0^t e^{i(t-\tau)\Delta}(fg)(\tau)d\tau\bigg\rVert_{{s,b}}&\lesssim \lVert f\rVert_{{s,b}}\lVert (\sqrt{-\Delta})^{5(2b-1)+\mu}g\rVert_{L_x^2L_t^2}.\label{eq-AA3}
\end{align}
for every $f$ and $g$ representable as
\begin{align*}
f(t,x)&=\sum_{n>N,m\in\mathbb{Z}} \widehat{f}(n,m)e_n(x)e(mt),\\
g(t,x)&=\sum_{n\leq N,m\in\mathbb{Z}} \widehat{g}(n,m)e_n(x)e(mt).
\end{align*}
\end{proposition}

\begin{proof}
It suffices to establish (\ref{eq-AA3}) for $s=0$.  Since 
\begin{align*}
\lVert DF\rVert_{0,b}\sim \lVert (-\Delta)^{1/2}F\rVert_{0,b}=\lVert F\rVert_{1,b}
\end{align*}
the statement then clearly follows for $s=1$ and we can then interpolate.

Let $\mu>0$ be given.  We begin by writing 
\begin{align*}
g=\sum_{K<N} g_K\quad\textrm{with}\quad g_K(t,x)=\sum_{\substack{n\sim K\\m\in\mathbb{Z}}} \widehat{g}(m,n)e_n(x)e(mt).
\end{align*}

Fix $K\geq 1$ a dyadic integer, and estimate
\begin{align}
\bigg\lVert \int_0^t e^{i(t-\tau)\Delta}fg_K(\tau)d\tau\bigg\rVert_{{0,b}}.\label{eq313}
\end{align}

Set $K_1=K^5$ and let $\mathcal{J}$ denote a partition of $\mathbb{Z}$ in intervals $I$ of size $K_1$.  Write
\begin{align*}
fg_K&=\sum_{\substack{I,I'\in \mathcal{J}\\\dist(I,I')\leq K_1}} P_{I'}(P_Ifg_K)+\sum_{\substack{I,I'\in \mathcal{J}\\\dist(I,I')>K_1}}P_{I'}(P_Ifg_K)\\
&=:(I)+(II).
\end{align*}

By construction, the contribution of (I) in ($\ref{eq313}$) is clearly bounded by
\begin{align}
\left[\sum_{I\in\mathcal{J}} \left\lVert P_I\left(\int_0^t e^{i(t-\tau)\Delta}(P_{\widetilde{I}}fg_K)(\tau)d\tau\right)\right\rVert_{0,b}^2\right]^{1/2}\label{eq-hw-36}
\end{align}
with $\widetilde{I}=I$ or $\widetilde{I}\in \mathcal{J}$ a neighbor of $I$.  Applying Lemma $\ref{lem35}$ to each term of ($\ref{eq-hw-36}$) gives the estimate
\begin{align*}
K_1^{2b-1+\epsilon}\left(\sum_{I\in\mathcal{J}}\lVert P_{\widetilde{I}}fg_K\rVert_{L_x^{\frac{4}{3}+\epsilon}L_t^\frac{4}{3}}^2\right)^{1/2}\leq K_1^{2b-1+\epsilon}\lVert g_K\rVert_{L_x^{2+\epsilon'}L_t^2}\left(\sum_{I\in\mathcal{J}}\lVert P_If\rVert_{L_x^{4-\epsilon}L_t^4}^2\right)^{1/2}
\end{align*}
and since by Lemma $\ref{lem26}$
\begin{align*}
\lVert P_If\rVert_{L_x^{4-\epsilon}L_t^4}\leq K_1^\epsilon \lVert P_If\rVert_{0,b}
\end{align*}
we obtain
\begin{align}
K_1^{2b-1+2\epsilon+\epsilon'}\lVert f\rVert_{0,b}\lVert g_K\rVert_{L_{x,t}^2}\leq K^{5(2b-1)+\mu}\lVert f\rVert_{0,b}\lVert g_K\rVert_{L_{x,t}^2}.\label{eq-hw-37}
\end{align}

Next we estimate the contribution of (II) in ($\ref{eq313}$), which will appear as an error term.  This contribution may certainly be bounded by $\lVert (II)\rVert_{L_{x,t}^2}$.  Fix $t$ and write $f(x,t)=\sum_{n\geq 1} f_ne_n(x)$.  Clearly, we obtain by duality
\begin{align}
\nonumber \left\lVert \sum_{\substack{I,I'\in\mathcal{J}\\\dist(I,I')>K_1}} P_{I'}(P_Ifg_K)\right\rVert_{L_x^2}&\leq \max_{\sum_{n\geq 1} |a_n|^2\leq 1} \left(\sum_{|n-n'|>K_1} |f_n|\, |a_{n'}|\, |\langle e_ng_K,e_{n'}\rangle|\right)^{1/2}\\
&=\sum_{n,n'\geq 1} |f_n|\, |a_{n'}|\, M_{n,n'}\label{eq-hw-38}
\end{align}
where we have denoted
\begin{align}
M_{n,n'}=|\langle e_ng_K,e_{n'}\rangle|\chi_{|n-n'|>K_1}.\label{eq-hw-39}
\end{align}

We therefore need to estimate the norm $\lVert M\rVert_{\ell^2(\mathbb{Z}_+)\rightarrow\ell^2(\mathbb{Z}_+)}$ of the matrix $M$.  We will rely on the Shur bound 
\begin{align*}
\lVert M\rVert\leq \sup_n\left(\sum_{n'}|M_{n,n'}|\right)
\end{align*}
(since $M$ is symmetric).

It remains to bound $\langle e_ng_K,e_{n'}\rangle$.  Write
\begin{align*}
|z_n^2-z_{n'}^2|\, |\langle e_ng_K,e_{n'}\rangle|=|\langle \Delta e_ng_K,e_{n'}\rangle-\langle e_ng_K,\Delta e_{n'}\rangle|
\end{align*}
and integrate by parts, using the Dirichlet boundary condition on $B_2$ to obtain the bound
\begin{align*}
|\langle \nabla e_n\cdot \nabla g_K,e_{n'}\rangle|+|\langle e_n\Delta g_K,e_{n'}\rangle|.
\end{align*}

Hence, for $|n-n'|\gg 1$, we have
\begin{align}
|\langle e_ng_K,e_{n'}\rangle|\lesssim (n+n')^{-1}|n-n'|^{-1}(|\langle \nabla e_n\cdot \nabla g_K,e_{n'}\rangle|+|\langle e_n\Delta g_K,e_{n'}\rangle|).\label{eq-hw-40}
\end{align}

Therefore, fixing $n\in\mathbb{Z}_+$,
\begin{align*}
\sum_{n'} M_{n,n'}&\lesssim \sum_{\{s:2^s>K_1\}} 2^{-s}\sum_{\{n:|n-n'|\sim 2^s\}} (n+n')^{-1}(|\langle \nabla e_n\cdot \nabla g_K,e_{n'}\rangle|\\
&\hspace{2.6in}+|\langle e_n\Delta g_K,e_{n'}\rangle|)\\
&\lesssim \sum_{\{s:2^s>K_1\}} \frac{2^{-s/2}}{n}\left[\sum_{n'\geq 1} (|\langle \nabla e_n\cdot \nabla g_K,e_{n'}\rangle|^2+|\langle e_n\Delta g_K,e_{n'}\rangle|^2)\right]^{1/2}\\
&\lesssim \frac{K_1^{-1/2}}{n}(\lVert \nabla e_n\cdot \nabla g_K\rVert_{L_x^2}+\lVert e_n\Delta g_K\rVert_{L_x^2})
\end{align*}
where we used Cauchy-Schwarz and Parseval.  The above quantity is then bounded by a multiple of 
\begin{align*}
\frac{K_1^{-1/2}}{n}(\lVert \nabla e_n\rVert_{L_x^2}\lVert \nabla g_K\rVert_{L_x^\infty}+\lVert \Delta g_K\rVert_{L_x^2})&\lesssim \frac{K^2}{K_1^{1/2}}\lVert g(t)\rVert_{L_x^2}<K^{-1}\lVert g(t)\rVert_{L_x^2}
\end{align*}
where we have used the choice of $K_1$.

This proves that
\begin{align*}
(\ref{eq-hw-38})&\lesssim \frac{1}{K}\lVert g(t)\rVert_{L_x^2}\left(\sum_{n\geq 1} |f_n|^2\right)^{1/2}=\frac{1}{K}\lVert f(t)\rVert_{L_x^2}\lVert g(t)\rVert_{L_x^2}
\end{align*}
and
\begin{align}
\lVert (II)\rVert_{L_{x,t}^2}&\lesssim \frac{1}{K}\lVert f\rVert_{L_{x,t}^2}\lVert g\rVert_{L_{x,t}^2}.\label{eq-hw-41}
\end{align}

Summing (\ref{eq-hw-37}), (\ref{eq-hw-41}) over dyadic $K$ gives the estimate
\begin{align*}
\lVert f\rVert_{0,b}\lVert (\sqrt{-\Delta})^{5(2b-1)+\mu}g\rVert_{L_{x,t}^2}
\end{align*}
as desired.
\end{proof}

\begin{remark}
Note that the second factor on the right-hand side of ($\ref{eq-AA3}$) involves a classical space-time norm (rather than $X^{s,b}$-type norms).  This will be important in the next section.
\end{remark}

\section{Proof of Theorem $\ref{thm-nls-2d}$: convergence of the solutions of the truncated equations}

In this section we establish Theorem $\ref{thm-nls-2d}$ for radial NLS on the $2D$ ball and prove convergence of the sequence $(u_{N})$ in the space $X^{s,b}$, almost surely in $\omega\in \Omega$.  Convergence in the space $C_tH_x^s$ then follows from the embedding $X^{s,b}\hookrightarrow L_t^\infty H_x^s$.

\begin{proof}[Proof of Theorem $\ref{thm-nls-2d}$.] 
Let $0<s<\frac{1}{2}$ and $T>0$ be given.  We first establish the $\mu_F$-almost sure convergence of solutions $(u_{N_k})$ with $N_k=2^k$.  Moreover, up to a covering argument partitioning the time interval, we may assume without loss of generality that $T<\frac{1}{2}$.

Let $\sigma\in (0,\frac{1}{2})$, $1\leq r<\frac{2}{\sigma}$ and $1\leq p,q<\infty$ be fixed parameters to be determined later in the argument.  Fix $N_0<N_1$, and for each $\omega\in \Omega$, let $u_{N_0}$, $u_{N_1}$ be the solutions of the truncated equation ($\ref{truncated}$) with corresponding initial data $P_{N_0}\phi^{(\omega)}$ and $P_{N_1}\phi^{(\omega)}$.  

Let $B_{N_0}>0$ also be a parameter to be determined later satisfying $B_{N_0}\lesssim N_0^\gamma$ for some $\gamma>0$.  Then, invoking ($\ref{eq-aaaq-1}$) and Proposition $\ref{prop-prob}$, there exists a set $\Omega(N_0,N_1)$ with
\begin{align}
\mu_F(\Omega(N_0,N_1))\lesssim \exp(-B_{N_0}^c)\label{prob-omega}
\end{align}
such that for all $\omega\in \Omega\setminus \Omega(N_0,N_1)$ one has the bounds
\begin{align}
\lVert P_{N_1}\phi^{(\omega)}-P_{N_0}\phi^{(\omega)}\rVert_{H_x^s}&\lesssim N_0^{s-\frac{1}{2}},\label{constr-bd-1}
\end{align}
together with
\begin{align}
\max\big\{\lVert u_{N_0}\rVert_{L_x^pL_t^q}\, ,\,\lVert u_{N_1}\rVert_{L_x^pL_t^q}\big\}&< B_{N_0},
\end{align}
and
\begin{align}
\max\big\{\lVert (\sqrt{-\Delta})^{\sigma}u_{N_0}\rVert_{L_x^rL_t^q}\, ,\,\lVert (\sqrt{-\Delta})^{\sigma}u_{N_1}\rVert_{L_x^rL_t^q}\big\}&< B_{N_0}.\label{prob-omega-2}
\end{align}

Fixing $\omega\in \Omega\setminus \Omega(N_0,N_1)$, we now estimate the $X^{s,b}([0,T])$ norm of the difference $u_{N_1}-u_{N_0}$.  For this, we will use an iterative argument on short time intervals.  In particular, fixing a small value $\eta=\eta(N_0)>0$, and partitioning the interval $[0,T]$ into $T/\eta$ intervals $[t_i,t_{i+1})$, with $t_{i+1}-t_i=\eta$, we write
\begin{align}
\nonumber u_{N_1}(t)-u_{N_0}(t)&=e^{i(t-t_i)\Delta}(u_{N_1}(t_i)-u_{N_0}(t_i))\\
\nonumber &\hspace{0.2in}-i\int_{t_i}^{t} e^{i(t-\tau)\Delta}\Big[P_{N_1}(|u_{N_1}|^\alpha u_{N_1})(\tau)-P_{N_0}(|u_{N_1}|^\alpha u_{N_1})(\tau)\Big]d\tau\\
&\hspace{0.2in}-i\int_{t_i}^{t} e^{i(t-\tau)\Delta}P_{N_0}\Big[|u_{N_1}|^\alpha u_{N_1}-|u_{N_0}|^\alpha u_{N_0}(\tau)\Big]d\tau\label{eq21}
\end{align}
for $t\in [t_i,t_{i+1})$.  

We now estimate the $X^{s,b}$ norms of each term in ($\ref{eq21}$).  In what follows, all $X^{s,b}$ and $L_t^p$ norms will be taken on the time interval $[t_i,t_{i+1})$, unless otherwise indicated.  

Using the unitarity of the linear propagator and the definition of the $X^{s,b}$ norm, the first term is estimated as 
\begin{align}
\lVert e^{i(t-t_i)\Delta}(u_{N_1}(t_i)-u_{N_0}(t_i))\rVert_{{s,b}}\lesssim \lVert u_{N_1}(t_i)-u_{N_0}(t_i)\rVert_{H_x^s}.\label{eq36}
\end{align}

On the other hand, to estimate the second term in ($\ref{eq21}$), we fix $s'\in (s,1/2)$ and invoke Lemma $\ref{lem35}$, and the fractional product rule, which give the bound
\begin{align}
\nonumber &\bigg\lVert \int_{t_i}^{t} e^{i(t-\tau)\Delta}[P_{N_1}(|u_{N_1}|^\alpha u_{N_1})(\tau)-P_{N_0}(|u_{N_1}|^\alpha u_{N_1})(\tau)]d\tau\bigg\rVert_{{s,b}}\\
\nonumber &\hspace{0.2in}\lesssim N_0^{-(s'-s)}\lVert (\sqrt{-\Delta})^{2b-1+s'+\epsilon} |u_{N_1}|^\alpha u_{N_1}\rVert_{L_x^{\frac{4}{3}+\epsilon}L_t^{\frac{4}{3}}}\\
\nonumber &\hspace{0.2in}\lesssim N_0^{-(s'-s)}\lVert u_{N_1}\rVert_{L_x^{r_1}L_t^{\frac{4(\alpha+1)}{3}}}^\alpha \lVert (\sqrt{-\Delta})^{2b-1+s'+\epsilon}u_{N_1}\rVert_{L_x^{r_2}L_t^{\frac{4(\alpha+1)}{3}}}\\
&\hspace{0.2in}\lesssim N_0^{-(s'-s)}B_{N_0}^{\alpha+1},\label{eq37}
\end{align}
provided that $b$ is chosen sufficiently close to $\frac{1}{2}$, $\epsilon>0$ is chosen sufficiently small, and $\sigma$ is chosen large enough to ensure that
\begin{align*}
2b-1+s'+\epsilon<\sigma,
\end{align*}
and the values $r_1,r_2\geq 1$ satisfy $r_1\leq p$ and $r_2\leq r$ together with $\frac{4(\alpha+1)}{3}<q$ and 
\begin{align*}
\frac{3}{4+3\epsilon}=\frac{\alpha}{r_1}+\frac{1}{r_2}.
\end{align*}

Combining the estimates ($\ref{eq36}$) and ($\ref{eq37}$), we obtain
\begin{align}
\nonumber &\lVert u_{N_1}-u_{N_0}\rVert_{X^{s,b}([t_i,t_i+\eta))}\\
\nonumber &\hspace{0.6in}\lesssim \lVert u_{N_1}(t_i)-u_{N_0}(t_i)\rVert_{H_x^{s}}+N_0^{-(s'-s)}B_{N_0}^{\alpha+1}\\
\label{eqb_aa1}&\hspace{1.0in}+\bigg\lVert \int_{t_i}^{t} e^{i(t-\tau)\Delta}P_{N_0}\Big[|u_{N_1}|^\alpha u_{N_1}-|u_{N_0}|^\alpha u_{N_0}(\tau)\Big]d\tau\bigg\rVert_{{s,b}}
\end{align}
On the other hand, from the assumption that $\alpha$ is an even integer we obtain the expansion
\begin{align*}
&|u_{N_1}|^\alpha u_{N_1}-|u_{N_0}|^\alpha u_{N_0}\\
&\hspace{0.2in}=(u_{N_1}-u_{N_0})F_+(u_{N_0},u_{N_1},\overline{u_{N_0}},\overline{u_{N_1}})+(\overline{u_{N_1}}-\overline{u_{N_0}})F_-(u_{N_0},u_{N_1},\overline{u_{N_0}},\overline{u_{N_1}})
\end{align*}
with $F_+$, $F_-$ homogeneous polynomials of degree $\alpha$.  

We will only estimate the $F_+$ term; the estimate for the $F_-$ term is identical.  Performing dyadic decompositions in frequency, we obtain
\begin{align}
\nonumber &\bigg\lVert \int_{t_i}^t e^{i(t-\tau)\Delta}P_{N_0}[(u_{N_1}-u_{N_0})F_+](\tau)d\tau\bigg\rVert_{{s,b}}\\
\nonumber &\hspace{0.2in}\lesssim \sum_K\bigg\lVert \int_{t_i}^t e^{i(t-\tau)\Delta}P_{N_0}\bigg[\bigg(P_{>K}(u_{N_1}-u_{N_0})\bigg)P_{K<\cdot\leq 2K}F_+\bigg](\tau)d\tau\bigg\rVert_{{s,b}}\\
\nonumber &\hspace{0.4in}+\sum_K \bigg\lVert \int_{t_i}^t e^{i(t-\tau)\Delta}P_{N_0}\bigg[\bigg(P_{\leq K}(u_{N_1}-u_{N_0})\bigg)P_{K<\cdot\leq 2K}F_+\bigg](\tau)d\tau\bigg\rVert_{{s,b}}\\
\label{eqb_aa2}&\hspace{0.2in}=:\sum_K (I)_K+(II)_K,
\end{align}

To estimate the terms $(I)_K$, fix $\epsilon>0$ small and note that by applying Proposition $\ref{propP0}$ followed by the fractional product rule we obtain
\begin{align}
\nonumber &\bigg\lVert \int_{t_i}^t e^{i(t-\tau)\Delta}P_{N_0}\bigg[\bigg(P_{>K}(u_{N_1}-u_{N_0})\bigg)P_{K<\cdot\leq 2K}F_+\bigg](\tau)d\tau\bigg\rVert_{{s,b}}\\
\nonumber &\hspace{0.2in}\lesssim K^{-\epsilon}\lVert u_{N_1}-u_{N_0}\rVert_{{s,b}}\lVert (\sqrt{-\Delta})^{5(2b-1)+2\epsilon}P_{K<\cdot\leq 2K}F_+\rVert_{L_{t,x}^2}\\
\nonumber &\hspace{0.2in}\lesssim K^{-\epsilon}\eta^{1/4}\lVert u_{N_1}-u_{N_0}\rVert_{{s,b}}\lVert u\rVert_{L_{t,x}^{8(\alpha-1)}}^{\alpha-1}\lVert (\sqrt{-\Delta})^{5(2b-1)+2\epsilon}u\rVert_{L_{t,x}^{8}}\\
\label{eqb_aa4}&\hspace{0.2in}\lesssim K^{-\epsilon}\eta^{1/4}B_{N_0}^{\alpha}\lVert u_{N_1}-u_{N_0}\rVert_{{s,b}}
\end{align}
for $\epsilon>0$ sufficiently small and $b>\frac{1}{2}$ close enough to $\frac{1}{2}$ to ensure $8<\frac{2}{5(2b-1)+\epsilon}$.

Turning to $(II)_K$, we argue in a similar manner.  In particular, again fix $\epsilon>0$ small and note that by Lemma $\ref{lem35}$, the H\"older inequality and the fractional product rule, one has
\begin{align}
\nonumber &\bigg\lVert \int_{t_i}^t e^{i(t-\tau)\Delta}P_{N_0}\bigg[\bigg(P_{\leq K}(u_{N_1}-u_{N_0})\bigg)P_{K<\cdot\leq 2K}F_+\bigg](\tau)d\tau\bigg\rVert_{{s,b}}\\
\nonumber &\hspace{0.2in}\lesssim K^{-\epsilon}\lVert u_{N_1}-u_{N_0}\rVert_{L_{t}^2L_{x}^2}\lVert (\sqrt{-\Delta})^{2b-1+s+2\epsilon}F_+\rVert_{L_{t,x}^4}\\
\nonumber &\hspace{0.2in}\lesssim K^{-\epsilon}\eta^{1/2}\lVert u_{N_1}-u_{N_0}\rVert_{L_t^\infty L_x^2}\lVert u\rVert_{L_{x}^\frac{4(4+\epsilon')(\alpha-1)}{\epsilon'}L_t^{{8(\alpha-1)}}}^{\alpha-1}\lVert (\sqrt{-\Delta})^{2b-1+s+2\epsilon}u\rVert_{L_{x}^{4+\epsilon'}L_t^8}\\
\label{eqb_aa3}&\hspace{0.2in}\lesssim K^{-\epsilon}\eta^{1/2}B_{N_0}^\alpha\lVert u_{N_1}-u_{N_0}\rVert_{{s,b}}
\end{align}
for every dyadic $K\geq 1$, provided the values $\epsilon$, $\epsilon'$ and $b$ are chosen sufficiently small to ensure $4+\epsilon'<\frac{2}{2b-1+s+2\epsilon}$.

Combining ($\ref{eqb_aa1}$), ($\ref{eqb_aa2}$), ($\ref{eqb_aa4}$) and ($\ref{eqb_aa3}$) and evaluating the summation over $K$ then gives the bound
\begin{align}
\nonumber \lVert u_{N_1}-u_{N_0}\rVert_{{s,b}}&\lesssim \lVert u_{N_1}(t_i)-u_{N_0}(t_i)\rVert_{H_x^{s}}+N_0^{-(s'-s)}B_{N_0}^{\alpha+1}\\
&\hspace{0.4in}+\eta^{1/4}B_{N_0}^\alpha\lVert u_{N_1}-u_{N_0}\rVert_{{s,b}},\label{4.19}
\end{align}
so that choosing 
\begin{align*}
\eta= cB_{N_0}^{-4\alpha}\quad\textrm{and}\quad B_{N_0}=(c\log N_0)^{\frac{1}{4\alpha}}
\end{align*}
with $c>0$ sufficiently small (depending on the implicit constant) gives the estimate
\begin{align}
\lVert u_{N_1}-u_{N_0}\rVert_{{s,b}}&\lesssim \lVert u_{N_1}(t_i)-u_{N_0}(t_i)\rVert_{H_x^{s}}+N_0^{-(s'-s)/2}.\label{constr-bd-iterative}
\end{align}

It now remains to estimate the $H_x^s$ norm appearing on the right side of ($\ref{constr-bd-iterative}$).  We argue iteratively, recalling the initial-time bound on $u_{N_1}-u_{N_0}$ given by ($\ref{constr-bd-1}$) and successively applying ($\ref{constr-bd-iterative}$) to yield the bound
\begin{align}
\lVert u_{N_1}-u_{N_0}\rVert_{L_t^\infty([0,T);H_x^s)}&\lesssim C^{T/\eta}N_0^{-(s'-s)/4}.\label{constr-bd-iterative-2}
\end{align}

Substituting ($\ref{constr-bd-iterative-2}$) into ($\ref{constr-bd-iterative}$) and using the above choice of $B_{N_0}$, we obtain the estimate
\begin{align}
\lVert u_{N_1}-u_{N_0}\rVert_{X^{s,b}([t_i,t_{i+1}))}&\lesssim N_0^{-(s'-s)/4}\label{constr-bd-iterative-3}
\end{align}
for all subintervals $[t_i,t_{i+1}]\subset [0,T]$.  By a covering argument, one immediately gets
\begin{align}
\lVert u_{N_1}-u_{N_0}\rVert_{X^{s,b}([0,T])}&\lesssim_T N_0^{-(s'-s)/8}\label{constr-bd-iterative-4}
\end{align}
for all initial data $\phi^{(\omega)}$ with $\omega\in \Omega\setminus \Omega(N_0,N_1)$.  

Now, letting $N_k=2^k$ and setting
\begin{align*}
\Omega_0:=\bigcap_{J\geq 1}\bigcup_{j\geq J} \Omega(N_j,N_{j+1}),
\end{align*}
we obtain that $(u_{N_k})$ is a Cauchy sequence in $X^{s,b}([0,T])$ for all $\omega\in \Omega\setminus \Omega_0$.  Moreover, recalling the bound ($\ref{prob-omega}$), we have 
\begin{align}
\mu_F(\Omega_0)&\lesssim \sum_{j\geq J} \exp(-c\log N_j^{c/4\alpha})\label{conv-summation}
\end{align}
for all $J\geq 1$, and thus $\mu_F(\Omega_0)=0$.  The sequence $(u_{N_k})$ then converges in $X^{s,b}([0,T])$  $\mu_F$-almost surely.  

In order to conclude convergence of the full sequence $(u_N)$ a slightly more refined analysis is required.  The essential difficulty is a consequence of the dependence of the excluded sets of initial data $\Omega(N_j,N_{j+1})$ on $N_j$ and $N_{j+1}$, since without passing to the subsequence $(u_{N_k})$ we cannot immediately conclude the convergence on the right side of ($\ref{conv-summation}$).  

In this case, fix a parameter $C(p)\gg 1$ to be determined, and note that for each $N_0\gg 1$, we may consider
\begin{align*}
M=(\log N_0)^{C(p)}
\end{align*}
and replace the set $\Omega\setminus \Omega(N_0,N_1)$ chosen in (\ref{prob-omega})--(\ref{prob-omega-2}) by
\begin{align*}
\Omega'(N_0)=\bigg\{\omega\in\Omega&:\lVert \phi^{(\omega)}-P_{N_0}\phi^{(\omega)}\rVert_{H_x^s}\lesssim N_0^{s-\frac{1}{2}},\quad \lVert u_{N_0}\rVert_{L_x^pL_t^q}<B_{N_0},\\
&\hspace{0.2in}\lVert (\sqrt{-\Delta})^\sigma u_{N_0}\rVert_{L_x^rL_t^q}<B_{N_0}, \\
&\hspace{0.2in}\max_{N_0\leq N<2N_0} \lVert u_N-P_Mu_N\rVert_{L_x^pL_t^q}<1,\\
&\hspace{0.2in}\max_{N_0\leq N<2N_0} \lVert (\sqrt{-\Delta})^\sigma (u_N-P_Mu_N)\rVert_{L_x^rL_t^q}<1\bigg\}
\end{align*}

Recalling Proposition $\ref{prop-prob}$ and Proposition $\ref{prop-prob-2}$, we then have
\begin{align*}
\mu_F(\Omega\setminus \Omega'(N_0))<e^{-(B_{N_0})^c}+2N_0e^{-(M^{\frac{2}{p}-\sigma})^c}
\end{align*}
so that by choosing $C(p)$ sufficiently large we obtain
\begin{align*}
\mu_F(\Omega\setminus \Omega'(N_0))<3e^{-(B_{N_0})^c}.
\end{align*}

Now, note that for $N_0\leq N_1<2N_0$ we have
\begin{align*}
\lVert u_{N_1}\rVert_{L_x^pL_t^q}&\leq \lVert u_{N_1}-P_Mu_{N_1}\rVert_{L_x^pL_t^q}+\lVert P_M(u_{N_1}-u_{N_0})\rVert_{L_x^pL_t^q}\\
&\hspace{0.2in}+\lVert u_{N_0}-P_Mu_{N_0}\rVert_{L_x^pL_t^q}+\lVert u_{N_0}\rVert_{L_x^pL_t^q}\\
&\leq 2+T^{1/q}\lVert P_M(u_{N_1}-u_{N_0})\rVert_{L_{t,x}^\infty}+B_{N_0}\\
&\leq 2+T^{1/q}M^{1-s}\lVert u_{N_1}-u_{N_0}\rVert_{s,b}+B_{N_0}\\
&\leq (\log N_0)^C\lVert u_{N_1}-u_{N_0}\rVert_{s,b}+2B_{N_0}
\end{align*}
and
\begin{align*}
&\lVert (\sqrt{-\Delta})^\sigma u_{N_1}\rVert_{L_x^rL_t^q}\\
&\hspace{0.2in}\leq \lVert (\sqrt{-\Delta})^\sigma (u_{N_1}-P_Mu_{N_1})\rVert_{L_x^rL_t^q}+\lVert (\sqrt{-\Delta})^\sigma P_M(u_{N_1}-u_{N_0})\rVert_{L_x^rL_t^q}\\
&\hspace{0.4in}+\lVert (\sqrt{-\Delta})^\sigma (u_{N_0}-P_Mu_{N_0})\rVert_{L_x^rL_t^q}+\lVert (\sqrt{-\Delta})^\sigma u_{N_0}\rVert_{L_x^rL_t^q}\\
&\hspace{0.2in}\leq 2+T^{1/q}M^\sigma\lVert P_M(u_{N_1}-u_{N_0})\rVert_{L_{t,x}^\infty}+B_{N_0}\\
&\hspace{0.2in}\leq 2+T^{1/q}M^{1+\sigma-s}\lVert P_M(u_{N_1}-u_{N_0})\rVert_{s,b}+B_{N_0}\\
&\hspace{0.2in}\leq (\log N_0)^C\lVert u_{N_1}-u_{N_0}\rVert_{s,b}+2B_{N_0}.
\end{align*}

For $\omega\in \Omega'(N_0)$, the analogue of ($\ref{4.19}$) then becomes
\begin{align*}
\lVert u_{N_1}-u_{N_0}\rVert_{s,b}&\lesssim \lVert u_{N_1}(t_i)-u_{N_0}(t_i)\rVert_{H_x^{s}}+\eta^{1/4}B_{N_0}^\alpha\lVert u_{N_1}-u_{N_0}\rVert_{{s,b}}\\
&\hspace{0.4in}+(\log N_0)^{C\alpha}\lVert u_{N_1}-u_{N_0}\rVert_{s,b}^{\alpha+1}+N_0^{-(s'-s)}B_{N_0}^{\alpha+1}.
\end{align*}

It then follows that
\begin{align*}
\lVert u_{N_1}-u_{N_0}\rVert_{s,b}&\lesssim \lVert u_{N_1}(t_i)-u_{N_0}(t_i)\rVert_{H_x^s}+(\log N_0)^{C\alpha}\lVert u_{N_1}-u_{N_0}\rVert_{s,b}^{\alpha+1}\\
&\hspace{0.2in}+N_0^{-\frac{1}{2}(s'-s)}
\end{align*}
and thus (\ref{constr-bd-iterative})--(\ref{constr-bd-iterative-4}) hold as before.
\end{proof}

\begin{remark}
As we pointed out, our assumption that $\alpha\in 2\mathbb{Z}_+$ in (\ref{equation}) is merely technical, though we would assume $\alpha\geq 2$ for smoothness reasons.  Alternatively, one could consider nonlinearities of the form
\begin{align}
F(u)=\mp(1+|u|^2)^{\alpha/2},\quad \alpha>0\label{eq-hw-1}
\end{align}
(cf. \cite{T2}).

The method described above may be carried out in higher dimension, leading to the analogue of Theorem $\ref{thm-nls-2d}$ for the radial defocusing NLS on $B_d$
\begin{align}
iu_t+\Delta u-(1+|u|^2)^{\alpha/2}u=0\label{eq-hw-2}
\end{align}
provided $\alpha<\frac{2}{d-2}$.

In $3$D, the counterpart of Theorem $\ref{thm-nls-2d}$ for the defocusing cubic NLS 
\begin{align}
iu_t+\Delta u-|u|^2u=0\label{eq-hw-3}
\end{align}
was established in \cite{BB-3} and seems to require a more delicate analysis.
\end{remark}

\section{The mass-critical focusing case}

In this section we complete the proof of Theorem $\ref{thm-focusing}$.  As noted in the introduction, the essential ingredient is contained in the following proposition, which shows that restriction to a sufficiently small $L_x^2$ ball leads to the construction of a well-defined Gibbs measure.

\begin{proposition}
\label{prop-focusing}
The measure
\begin{align*}
\exp\bigg(\int_{B_2} |\phi|^4dx\bigg)\chi_{\{\lVert \phi\rVert_{L_x^2}<\rho\}}(\phi)d\mu_F(\phi)
\end{align*}
is a bounded measure, provided that $\rho>0$ is chosen sufficiently small.
\end{proposition}

\begin{proof}
Fix $\rho>0$ to be determined.  To establish the proposition, it suffices to show that there exists $C>1$ such that for every $\lambda\gg 1$, we have
\begin{align}
\mu_F(\{\phi:\lVert \phi\rVert_{L_x^4}>\lambda,\,\lVert \phi\rVert_{L_x^2}<\rho\})\lesssim e^{-C\lambda^4}.\label{f-2}
\end{align}

This will be possible by choosing $\rho$ sufficiently small.  

To estimate the left-hand side of ($\ref{f-2}$), we begin by writing 
\begin{align}
\bigg\lVert \sum_{n\geq 1} \frac{g_n(\omega)}{z_n}e_n\bigg\rVert_{L_x^4}\leq \sum_{M\geq 1} \bigg\lVert \sum_{n\sim M}\frac{g_n(\omega)}{z_n}e_n\bigg\rVert_{L_x^4}\label{f-3}
\end{align}
where the summation in $M$ is taken over dyadic integers.  

For each $M\geq 1$ dyadic, choose $j\in\mathbb{Z}$ such that $M\sim 2^j(\lambda/\rho)^2$, and define
\begin{align*}
\sigma_M=\tfrac{1}{j^2}.
\end{align*}
Then
$$\sum_{M\geq 1} \sigma_M\lesssim 1$$
and therefore the condition 
\begin{align}
\bigg\lVert \sum_{n\geq 1} \frac{g_n(\omega)}{z_n}e_n\bigg\rVert_{L_x^4}>\lambda\label{f-A1}
\end{align}
implies 
\begin{align}
\bigg\lVert \sum_{n\sim M} \frac{g_n(\omega)}{z_n}e_n\bigg\rVert_{L_x^4}\gtrsim \sigma_M\lambda\label{f-A2}
\end{align}
for some $M\geq 1$.

To proceed, we consider an additional spatial decomposition, writing
\begin{align*}
\bigg\lVert \sum_{n\sim M}\frac{g_n(\omega)}{z_n}e_n\bigg\rVert_{L_x^4}^4&\leq \sum_{k\geq 0} \bigg\lVert \sum_{n\sim M} \frac{g_n(\omega)}{z_n}e_n\bigg\rVert_{L_x^4(|x|\sim 2^{k}M^{-1})}^4\\
&\hspace{0.4in}+\bigg\lVert \sum_{n\sim M}\frac{g_n(\omega)}{z_n}e_n\bigg\rVert_{L_x^4(|x|\leq M^{-1})}^4.
\end{align*}

The condition ($\ref{f-A1}$) implies that for some $0\leq k\lesssim \log M$ we have
\begin{align}
\bigg\lVert \sum_{n\sim M} \frac{g_n(\omega)}{z_n}e_n\bigg\rVert_{L_x^4(|x|\sim 2^kM^{-1})}\gtrsim \frac{\sigma_M\lambda}{(k+1)^{1/2}}\label{f-4}.
\end{align}

Let $(a_n)_{M\leq n<2M}$ be a sequence of arbitrary complex coefficients.  One then has, using the Bessel function asymptotics, the estimate
\begin{align*}
&\bigg\lVert \sum_{n\sim M} a_ne_n\bigg\rVert_{L_x^4(|x|\sim 2^kM^{-1})}^4\\
&\hspace{0.2in}\lesssim \int_{r\sim 2^kM^{-1}}\bigg|\sum_{n\sim M} a_ne^{iz_nr}\bigg|^4\frac{dr}{r}+\frac{1}{M^4}\left(\sum_{n\sim M} |a_n|\right)^4\int_{r\sim 2^kM^{-1}} \frac{dr}{r^5}\\
&\hspace{0.2in}\leq 2^{-k}M\int_0^t \bigg|\sum_{n\sim M}a_ne^{2\pi iz_nr}\bigg|^4dr+16^{-k}M^2\left(\sum_{n\sim M} |a_n|^2\right)^{2}\\
&\hspace{0.2in}\lesssim 2^{-k}M^2\bigg(\sum_{n\geq 1}|a_n|^2\bigg)^2,
\end{align*}
Taking $a_n=\frac{g_n(\omega)}{z_n}$ in this bound, it follows that if in addition to ($\ref{f-A1}$) we assume $\lVert \phi\rVert_{L_x^2}<\rho$, then
\begin{align}
\bigg\lVert \sum_{n\sim M} \frac{g_n(\omega)}{z_n}e_n\bigg\rVert_{L_x^4(|x|\sim 2^kM^{-1})}\lesssim 2^{-k/4}M^{1/2}\rho.\label{f-6}
\end{align}

Together with (\ref{f-4}), this implies that for $M$ and $k$ satisfying (\ref{f-A2}) and (\ref{f-4}) we have
\begin{align*}
M>\bigg(\frac{\lambda}{\rho}\bigg)^2\frac{2^{k/2}}{k+1}\sigma_M^2
\end{align*}
and thus
\begin{align*}
j^42^j>\frac{2^{k/2}}{k+1}
\end{align*}
so that $k\lesssim j$.  In particular, we can therefore let $j\geq 0$ when writing $M\sim 2^j\left(\frac{\lambda}{\rho}\right)^2$.

We now assemble the above ingredients into the desired probabilistic estimate (\ref{f-2}), for which we will make use of the estimate ($\ref{f-8}$) for Gaussian processes.  Noting that the bound $|e_n(x)|\lesssim |x|^{-1/2}$ implies 
\begin{align}
&\nonumber \mathbb{E}_\omega\bigg[\,\bigg\lVert \sum_{n\sim M} \frac{g_n(\omega)}{z_n}e_n\bigg\rVert_{L_x^4(|x|\sim 2^kM^{-1})}\,\bigg]\\
&\hspace{0.6in}\sim \bigg\lVert \bigg(\sum_{n\sim M} \frac{e_n^2}{n^2}\bigg)^{1/2}\bigg\rVert_{L_x^4(|x|\sim 2^kM^{-1})}\sim \bigg(\sum_{n\sim M} \frac{1}{n^2}\bigg)^{1/2}\sim \frac{1}{\sqrt{M}}\label{f-7}
\end{align}
apply ($\ref{f-8}$) with $$X(\omega)=\sum_{n\sim M} \frac{g_n(\omega)}{z_n}e_n\quad\textrm{and}\quad \lVert\,\cdot\,\rVert=\lVert\,\cdot\,\rVert_{L_x^4(|x|\sim 2^kM^{-1})}.$$  

By (\ref{f-7}), we have $\mathbb{E}[\lVert X\rVert]\sim \frac{1}{\sqrt{M}}$ and we take, according to (\ref{f-4})
\begin{align*}
t=\frac{\sqrt{M}\sigma_M\lambda}{(k+1)^{1/2}}\sim \frac{2^{j/2}\lambda^2}{(k+1)^{1/2}j^2\rho}\gtrsim 1
\end{align*}
since $k\lesssim j$.

We therefore conclude 
\begin{align}
\mathbb{P}_\omega\bigg[\bigg\lVert \sum_{n\sim M} \frac{g_n(\omega)}{z_n}e_n\bigg\rVert_{L_x^4(|x|\sim 2^kM^{-1})}\gtrsim \frac{\sigma_M\lambda}{(k+1)^{1/2}} \bigg]\lesssim e^{-c\frac{2^j}{(k+1)j^4}\rho^{-2}\lambda^4}.\label{f-9}
\end{align}

We now take the sum of (\ref{f-9}) over $j\geq 0$ and $0\leq k\lesssim j$, giving
\begin{align*}
\mu_F(\{\phi:\lVert \phi\rVert_{L_x^4}>\lambda,\,\lVert \phi\rVert_{L_x^2}<\rho\})&\leq e^{-c\rho^{-2}\lambda^4}<e^{-C\lambda^4}
\end{align*}
as desired for $\rho$ small enough.
\end{proof}

\begin{remark}
\label{rmk-hw-2-1}
We did not address here the issue of what is the optimal value of $\rho$ for Proposition $\ref{prop-focusing}$ to hold, which is an interesting question since it corresponds to the phase transition.  Note that this problem was not even settled for $d=1$ (cf. \cite{LRS}).
\end{remark}

\begin{remark}
In the $d$-dimensional setting, a similar argument applies for $\alpha=\frac{4}{d}$, providing normalized Gibbs measures
\begin{align}
e^{\frac{1}{\alpha+2}\lVert \phi\rVert_{L_x^{\alpha+2}}^{\alpha+2}}\chi_{\{\lVert \phi\rVert_{L_x^2}<\rho\}}\mu_F(d\phi)\label{eq-hw-2-1}
\end{align}
for $\rho$ sufficiently small (in the mass-subcritical case $\alpha<\frac{4}{d}$, $\rho$ can be taken arbitrarily).

Of course, one may replace the Hamiltonian by
\begin{align*}
\int_{B_d} \left[|\nabla\phi|^2-(1+|\phi|^2)^{\frac{\alpha}{2}+1}\right]dx
\end{align*}
with $\alpha$ as above, and consider the corresponding NLS
\begin{align}
iu_t+\Delta u+u(1+|u|^2)^{\alpha/2}=0.\label{eq-hw-2-2}
\end{align}

One obtains then the analogue of Theorem $\ref{thm-focusing}$, provided moreover $\alpha<\frac{2}{d-2}$ according to the comment at the end of $\S5$.
\end{remark}

\begin{remark}
\label{remark-s5}
Returning to Remark $\ref{rmk-hw-2-1}$, a similar phase transition occurs as for $d=1$, $\alpha=4$, described in \cite{LRS} (see also \cite{BS} for results on $\mathbb{T}^2$, though the situation there is different).

For sufficiently large $\rho$, the measures
\begin{align}
d\mu_G^{(N)}=e^{\frac{1}{4}\lVert \phi_N\rVert_{L_x^4}^4}\chi_{\{\lVert \phi_N\rVert_{L_x^2}<\rho\}}d\mu_F^{(N)}\label{eq-hw-3-1}
\end{align}
become unbounded.  A priori, this may not rule out the possibility that their normalization has an interesting limit distribution with perhaps a well-defined Schr\"odinger dynamics.  But this turns out not to be the case.  The distribution $\mu_G^{(N)}$ concentrates on functions $\phi_N$ for which
\begin{align}
\lVert \phi_N\rVert_{L_x^2\left(|x|<O\left(\frac{1}{N}\right)\right)}=O(\rho)\label{eq-hw-3-2}
\end{align}
and hence, in the limit, $|\phi_N|^2$ exhibits a delta function behavior at $x=0$.  

Note that for such functions
\begin{align}
\nonumber H(\phi_N)=\int_{B_2}\left[|\nabla \phi_N|^2-\frac{1}{4}|\phi_N|^4\right]&<CN^2\lVert \phi_N\rVert_{L_x^2}^2-cN^2\lVert \phi_N\rVert_{L_x^2\left(|x|<O\left(\frac{1}{N}\right)\right)}^4\\
&<-CN^2\rho^4<0\label{eq-hw-3-3}
\end{align}
for $\rho$ sufficiently large.  Invoking Kavian's extension of Glassey's theorem (see \cite{K}), it follows that the solution to the Cauchy problem
\begin{align*}
\left\lbrace\begin{array}{l}iu_t+\Delta u+u|u|^2=0\\
u|_{t=0}=\phi_N\end{array}\right.
\end{align*}
blows up in finite time.

Let us verify ($\ref{eq-hw-3-2}$), taking for $\rho$ a sufficiently large fixed constant.  We prove that
\begin{align}
\log \lVert \mu_G^{(N)}\rVert\sim \rho^4N^2\label{eq-hw-3-4}.
\end{align}

Since clearly $\lVert \phi_N\rVert_{L_x^4}\lesssim N^{1/2}\lVert \phi_N\rVert_{L_x^2}\leq \rho N^{1/2}$, the upper bound in ($\ref{eq-hw-3-4}$) is clear.  Conversely, write
\begin{align}
\lVert \mu_G^{(N)}\rVert>e^{\frac{c}{4}\rho^4N^2}\mu_F^{(N)}\left(\left\{\phi_N:\lVert \phi_N\rVert_{L_x^2}<\rho,\, \lVert \phi_N\rVert_{L_x^4}>c\rho N^{1/2}\right\}\right).\label{eq-hw-3-5}
\end{align}

Since $|\phi_N(0)|\lesssim N^{1/2}\lVert \phi_N\rVert_{L_x^4}$, we have
\begin{align}
\nonumber &\mu_F^{(N)}\left(\left\{\phi_N:\lVert \phi_N\rVert_{L_x^2}<\rho,\, \lVert \phi_N\rVert_{L_x^4}>c\rho N^{1/2}\right\}\right)\\
\nonumber &\hspace{0.2in}\geq \mes\left(\left\{\omega:\sum_{n=1}^N \frac{|g_n(\omega)|^2}{z_n^2}<\rho^2\quad\textrm{and}\quad \left|\sum_{n=1}^N\frac{g_n(\omega)}{z_n}e_n(0)\right|>c'\rho N\right\}\right)\\
\nonumber &\hspace{0.2in}\geq \mes\left(\left\{\omega:|g_1(\omega)|,\cdots,|g_{[N/2]}(\omega)|<1\quad\right.\right.\\
\nonumber &\hspace{1.2in}\left.\left.\textrm{and}\quad g_{[N/2]+1}(\omega),\cdots,g_N(\omega)\sim c''\rho \sqrt{N}\right\}\right)\\
&\hspace{0.2in}>e^{-C\rho^2N^2}.\label{eq-hw-3-6}
\end{align}
Hence, ($\ref{eq-hw-3-4}$) follows from ($\ref{eq-hw-3-5}$), ($\ref{eq-hw-3-6}$) by taking $\rho$ large enough.  

It is obvious from ($\ref{eq-hw-3-4}$) that $\mu_G^{(N)}$ is concentrated on functions $\phi_N$ for which
\begin{align}
\lVert \phi_N\rVert_{L_x^4}\sim \rho N^{1/2}.\label{eq-hw-3-7}
\end{align} 
Finally, assuming $\phi_N$ satisfies ($\ref{eq-hw-3-7}$), we verify ($\ref{eq-hw-3-2}$).  

Write
\begin{align*}
\lVert \phi_N\rVert_{L_x^4}^4&\leq \sum_{\{k:2^k<N\}}\int_{|x|\sim 2^k/N} |\phi_N(x)|^4dx\\
&\leq \sum_k \lVert \phi_N\rVert_{L_x^\infty(|x|\sim 2^k/N)}^2\lVert \phi_N\rVert_{L_x^2(|x|\sim 2^k/N)}^2
\end{align*}
and
\begin{align*}
\lVert \phi_N\rVert_{L_x^\infty(|x|\sim 2^k/N)}&\leq \sum_{n=1}^N |\widehat{\phi}_N(n)|\lVert e_n\rVert_{L_x^\infty(|x|\sim 2^k/N)}\\
&\lesssim \sqrt{\frac{N}{2^k}}\sum_{n=1}^N |\widehat{\phi}_N(n)|\\
&\lesssim N2^{-k/2}\lVert \phi_N\rVert_{L_x^2}\\
&\lesssim \rho 2^{-k/2}N.
\end{align*}
Hence, ($\ref{eq-hw-3-7}$) clearly implies ($\ref{eq-hw-3-2}$).
\end{remark}


\begin{thebibliography}{99}
\bibitem{B-GAFA1} J. Bourgain. Fourier transform restriction phenomena for certain lattice subsets and applications to nonlinear evolution equations. I. Schr\"odinger equations. Geom. Funct. Anal. 3 (1993), no. 2, 107-156.
\bibitem{B-GAFA2} J. Bourgain. Fourier transform restriction phenomena for certain lattice subsets and applications to nonlinear evolution equations. II. The KdV-equation. Geom. Funct. Anal. 3 (1993), no. 3, 209-262.
\bibitem{B-1} J. Bourgain. Periodic Korteweg De Vries Equations with Measures as Initial Data.  Sel. Math. New. Ser. 3 (1993), 115-159.
\bibitem{B1} J. Bourgain.  Periodic nonlinear Schr\"odinger equation in invariant measures. Comm. Math. Phys. 166 (1994), 1-24.
\bibitem{B2} J. Bourgain.  Invariant measures for the 2D-defocusing nonlinear Schr\"odinger equation. Comm. Math. Phys. 176 (1996), 421-445.
\bibitem{B3} J. Bourgain.  Invariant measures for the Gross-Piatevskii equation.  J. Math. Pures Appl. 76 (1997), p. 649--702.
\bibitem{B-PCMI} J. Bourgain.  Nonlinear Schr\"odinger equations, Hyperbolic Equations and Frequency Interactions, IAS/Park City Math. Ser., 5 (1999), Amer. Math. Soc., Providence, RI, p. 3-–157.
\bibitem{BB} J. Bourgain and A. Bulut. Gibbs measure evolution in radial nonlinear wave and Schr\"odinger equations on the ball.  Comptes Rendus Math. 350 (2012) 11-12, pp. 571--575.
\bibitem{BB-1} J. Bourgain and A. Bulut. Invariant Gibbs measure evolution for the 3D radial NLW on the ball. Preprint (2012).
\bibitem{BB-3} J. Bourgain and A. Bulut.  Almost sure global well posedness for the radial nonlinear Schr\"odinger equation on the unit ball II: the 3D case.  Preprint (2013). arXiv 1302.5409.
\bibitem{BS} D. Brydges and G. Slade.  Statistical Mechanics of the 2-Dimensional Focusing Nonlinear Schr\"odinger Equation.  Commun. Math. Phys. 182 (1996), 485--504.
\bibitem{BT12} N. Burq and N. Tzvetkov.  Random data Cauchy theory for supercritical wave equations. I. Local theory and II. A global existence result.  Invent. Math. 173 (2008), no. 3, 449-475 and 477-496.
\bibitem{K} O. Kavian.  A remark on the blowing-up of solutions to the Cauchy problem for nonlinear Schr\"odinger equations.  Trans. Amer. Math. Soc. 299 (1987), no. 1, 193--203.
\bibitem{LRS} J. Lebowitz, R. Rose and E. Speer.  Statistical mechanics of the nonlinear Schr\"odinger equation.  J. Stat. Phys. 50 (1988), 657--687.
\bibitem{T} N. Tzvetkov. Invariant measures for the defocusing nonlinear Schr\"odinger equation.  Annales de l'Institut Fourier, 58 (7) 2008, 2543--2604.
\bibitem{T2} N. Tzvetkov.  Invariant measures for the nonlinear Schrödinger equation on the disc. Dyn. Partial Differ. Equ. 3 (2006), no. 2, 111–-160.
\end{thebibliography}
\end{document}